\documentclass{amsart}

\usepackage{amsthm, amsfonts, amssymb, amsmath, latexsym, enumerate, times}
\usepackage{mathrsfs}
\usepackage{mathtools}

\usepackage{array}

\newtheorem{theorem}{Theorem}
\newtheorem{lemma}[theorem]{Lemma}
\newtheorem{claim}[theorem]{Claim}

\newtheorem{proposition}[theorem]{Proposition}

\theoremstyle{definition}

\newtheorem{remark}[theorem]{Remark}

\newcommand{\calE}{{\mathcal E}}
\newcommand{\calL}{{\mathcal L}}
\newcommand{\calM}{{\mathcal M}}

\newcommand{\calO}{{\mathcal O}}
\newcommand{\calX}{{\mathcal X}}
\newcommand{\calY}{{\mathcal Y}}
\newcommand{\calQ}{{\mathcal Q}}
\newcommand{\calC}{{\mathcal C}}
\newcommand{\bbP}{{\mathbb P}}
\newcommand{\bbR}{{\mathbb R}}
\newcommand{\bbZ}{{\mathbb Z}}

\DeclareMathOperator{\Pic}{Pic}
\DeclareMathOperator{\Uncoll}{Uncoll}

\newcommand{\Qperp}[1][s]{\calQ_{#1}^\perp}
\newcommand{\Nefperp}{\operatorname{Nef}^0}
\newcommand{\CK}[1][s]{\mathcal{CK}_{#1}}

\def\geq{\geqslant}
\def\leq{\leqslant}
\def\le{\leqslant}
\def\ge{\geqslant}

\begin{document}
 
\title[Quadric cones]{Quadric cones on the boundary of the Mori cone for very general blowups of the plane}

\author{Ciro Ciliberto}
\address{Dipartimento di Matematica, Universit\`a di Roma Tor Vergata, Via O. Raimondo
 00173 Roma, Italia}
\email{cilibert@axp.mat.uniroma2.it}

\author{Rick Miranda}
\address{Department of Mathematics, Colorado State University, Fort Collins (CO), 80523,USA}
\email{rick.miranda@colostate.edu}
 
\author{Joaquim Ro\'e}
\address{Departament de Matem\`atiques, Universitat Aut\`onoma de Barcelona, 08193 Bellaterra (Barcelona), Catalunya}
\email{joaquim.roe@uab.cat}

 
 
\maketitle

\begin{abstract} In this paper we show the existence of cones over a 8-dimensional rational spheres at the boundary of the Mori cone of the blow--up of the plane at $s\geq 13$ very general points. This gives evidence for De Fernex's strong $\Delta$--conjecture, which is known to imply Nagata's conjecture. This also implies the existence of a multitude of good and wonderful rays as defined in \cite  {CHMR13}.
\end{abstract}

\section*{Introduction} 
{
Fix a non-negative integer $s$ (usually we will assume that $s \geq 10$).
We denote by $X_s$ the blow--up of the complex projective plane
at $s$ very general points;
let $H$ be the divisor class of a general line,
let $E_i$ be the class of the exceptional divisor over the $i$-th point, {and 
let $N_s = \Pic(X_s)\otimes_{\bbZ}\bbR$,
where $\Pic(X_s)$ is the Picard group.
$N_s$ is a real vector space of dimension $s+1$
(with basis $H$ and the $E_i$).
The divisor class $L=dH-\sum_i m_i E_i$ represents the plane curves of degree $d$
having multiplicity at least $m_i$ at the $i$-th point.}

The non-negative real multiples of a nonzero vector {$L\in N_s$ is called a \emph{ray}, which we denote by $R=\langle L \rangle$}.
A ray is \emph{rational} if it contains an integral vector (in the $H,E_i$ basis).
A ray is \emph{effective} if it contains a (necessarily integral) vector representing an effective divisor class.
Neither the \emph{degree} of a ray (the coefficient of $H$)
nor the intersection number of two rays
are well--defined;
however the sign of these quantities are.
Therefore it makes sense for example to ask that a ray $R$ satisfy $\deg(R) >0$ or $R^2 > 0$;
for example, any effective ray must have non--negative degree.
By Riemann-Roch, if a rational ray has positive degree and self--intersection,
then it is effective.
A divisor or ray is \emph{nef} if it intersects all effective divisors non-negatively.

There are three cones in $N_s$ that are of interest for us.
The \emph{effective cone} is the cone generated by effective rays:
this is the cone of all finite linear combinations of effective divisor classes
with non--negative real coefficients.
The effective cone is, in general, not closed;
its closure is called the \emph{Mori} cone.
The dual of the Mori cone is the \emph{nef} cone (consisting of nef rays),
which is a closed cone as well.

Closed cones are defined by their extremal rays,
and so the identification of extremal rays for the Mori and nef cones 
is of fundamental importance.
Prior results constructing extremal rays of selfintersection zero for these cones
gave discrete examples (see \cite{CHMR13} and \cite{CMR23}) or families
of rays contained in the hyperplane orthogonal to $K_{X_s}=-3H+E_1+\dots+E_s$ (unpublished work by T.~de~Fernex for $s=10$ and by J.~C.~Ottem for $s=12$).
In this paper we will construct $9$-dimensional subsets (which are quadratic cones)
on the boundary of the Mori cone for all $s\ge 10$, and having positive intersection with $K_{X_s}$ if $s \geq 13$;
this illustrates the existence of both rational and irrational extremal rays.

A key lemma that produces extremal rays is provided by \cite[Lemma 3.8]{CHMR13},
which states that if a ray $R$ is rational, non--effective, and satisfies $\deg(R) \geq 0$ and $R^2 = 0$, then $R$ is nef, and is extremal for the Mori cone and the nef cone.
Such a ray is called \emph{good} in \cite{CHMR13};
in that paper we also defined a \emph{wonderful} ray, as one which 
is irrational and nef, with self--intersection zero.
Some wonderful rays are known;
the aforementiond examples of de Fernex and Ottem provide wonderful rays for $s=10$ and $s=12$, and the results of 
\cite{CMR23} imply the existence of wonderful rays for every $s \geq 10$.
Wonderful rays, being irrational and extremal,
prove that the Mori and nef cones are not rational polyhedral,
and provide evidence for stronger conjectures that we now describe.

Recall the canonical divisor $K_s = -3H+\sum_i E_i$ of $X_s$.
Define the \emph{de~Fernex ray} $F_s$
to be the ray generated by $\sqrt{s-1}H-\sum_{i=1}^s E_i$.
A ray $R$ is said to be \emph{de~Fernex positive, negative or orthogonal} 
according to $R\cdot F_s$ being positive, negative or null;
this terminology is parallel to a ray being $K_s$-positive, negative, or orthogonal.
The \emph{Strong $\Delta$-Conjecture}
(see \cite[Conjecture 3.10]{CHMR13})
is that if $s \geq 11$,
and $R$ is a rational de~Fernex non--positive ray of self-intersection zero,
then $R$ is not effective,
and therefore is a good ray.
(See \cite{DeF01}; there is a refinement for the $s=10$ case.)
Note that $R\cdot F_s \leq 0$ implies $R\cdot K_s>0$.

The Strong $\Delta$-Conjecture implies the Nagata Conjecture \cite{N59},
since it would imply that the \emph{Nagata ray} $\langle \sqrt{s}H-\sum_i E_i\rangle $
would be wonderful if $s$ is not a square.
It would also imply that, for de Fernex non-positive classes,
the boundary of the Mori cone is given by the classes with self--intersection zero.
Hence finding large subsets of the boundary of the Mori cone
given by such classes {is a strong measure of non-polyhedrality for the Mori cone and} provides intriguing evidence for the Strong $\Delta$-Conjecture.

De~Fernex's result in \cite{DeF01} that for $s=10$, 
all rays $R$ of self--intersection zero with $R\cdot K_{10}=0$ are nef means that 
this hyperplane section of the Mori cone does have a boundary given
by the quadratic equation $R^2=0$.
The boundary of the Mori cone of $X_{10}$ is so far unknown, but 
it is constrained by the nefness of all classes $L$ with $L^2=L\cdot K_{10}=0$, 
namely a cone over a 8-dimensional sphere.

In this paper we provide further $9$-dimensional quadratic cones
on the Mori cone boundary of $X_s$ for all $s \geq 13$, 
including some that constrain 
the de~Fernex--negative area of the cone.
We also provide a complete determination of the boundary classes orthogonal to $K_s$.

For clarity, because the points are assumed to be very general, specifying the $E_i$ is irrelevant; we  will therefore use the notation $L_d(m_1,\ldots, m_s)=dH-m_1E_1-\dots -m_sE_s$
for divisor classes in $\Pic(X_s)$, and we will use exponential notation for repeated multiplicities. Thus for example the canonical divisor can be written as $K_s=-L_3(1^s)$.
The corresponding linear system of singular curves on $\bbP^2$, sometimes identified with the projective space $\bbP(H^0(X_s,\calO_{X_{s}}(L)))$, will be denoted by
$\calL=\calL_d(m_1,\ldots, m_s)$.
Observe that $L$ is an effective class if and only if the (projective) dimension of $\calL$ is non-negative.}

Now we state the main results of this article.

\begin{theorem}\label{thm:main-kperp} 
	Let $s\ge 10$, and let $L$ be a class in $\Pic(X_s)$ with $L^2=K_s\cdot L=0$. The following are equivalent:
	\begin{enumerate}
		\item $L$ is nef.		
		\item There exist $s-10$ disjoint $(-1)$-curves $C_1, \dots, C_{s-10}$ 
such that $C_i\cdot L=0$ for each $i$.
		\item $L$ is equivalent, by the action of the Cremona-Kantor group (see \S \ref {sec:K-perp}),
{ to a multiple of $L_3(1^9,0^{s-9})$.}
	\end{enumerate}
	Moreover, for each fixed collection of $s-10$ disjoint $(-1)$-curves $C_1, \dots, C_{s-10}$, the rays spanned by all classes $L$ with $L^2=K\cdot L=C_i\cdot L=0 \,\forall i$, cover all rational points of a cone over a 8-dimensional rational sphere.
\end{theorem}

Here a cone over a $8$-dimensional rational sphere means a rational quadric of rank $10$ and signature $(1,9)$ contained in a $10$-dimensional rational linear subspace $\Pi$ of ${N_s}\cong \bbR_{s+1}$  and having rational points. In practice, the quadric will always be cut out on $\Pi$ by the equation $L^2=0$. 

Note that the second condition is empty for $s=10$, so in that case our statement is equivalent to De~Fernex's result mentioned above.

\begin{theorem}\label{thm:main-kplus} 
	Let $s\ge 13$. There exist 10-dimensional linear subspaces $\Pi\subset N_s$ such that
	the intersection of the Mori cone with $\Pi$ consists of classes $L\in \Pi$ such that $L^2\ge 0$ and $K_s\cdot L>0$, namely
	$$\Pi\cap \overline{NE}(X_s) = \{L\in \Pi \,|\, L^2\ge 0, K_s\cdot L>0\}. $$
\end{theorem}

\begin{theorem}\label{thm:main-fminus} 
	Let $s = k^2+4$ for some integer $k\ge 3$. 
	There exist 10-dimensional linear subspaces $\Pi\subset N_s$ such that
	the intersection of the Mori cone with $\Pi$ consists of classes $L\in \Pi$ such that $L^2\ge 0$ and has non--empty intersection with the $F_s$--negative half space. In particular, the conditions 
	$L\in \Pi$, $L^2\ge 0$ and $F_s\cdot L<0$ define an open subset (in the Euclidean topology) of an 8-dimensional sphere such that the cone over it is contained in the boundary of the Mori cone.
\end{theorem}

Our starting point is the de~Fernex result  on extremal
Mori rays orthogonal to $K_{10}$ (that we reprove in a different way in Section \ref{sec:K-perp} using Cremona maps), and then we proceed with the uncollision techniques
developed in \cite{CMR23} (see Section \ref{sec:collision})
to  produce new families of good and wonderful rays, in Section \ref{sec:round-parts}.

\section{Nefness on $K^\perp$}
\label{sec:K-perp}

In this section we study classes {$L=L_d(m_1,\ldots, m_s)$} such that $L^2 = L\cdot K_s=0$ with respect to nefness.
We recover by elementary methods de~Fernex's result for $s=10$ that every class with $L^2=L\cdot K_{10}=0$ is nef, and we extend it to describe the locus $\Nefperp$ of all nef classes with $L^2=L\cdot K_{s}=0$ for $s>10$.

Recall that, due to the Index theorem and the fact that $K_s^2>0$ for $s<9$, $K_s^2=0$ for $s=9$ and $K_s^2<0$ for $s>9$, the intersection form on the space $K_s^\perp$ of classes orthogonal to $K_s$ is negative definite for $s<9$, it is negative semidefinite for $s=9$, and it has signature $(1,s-1)$ for $s>9$. 
Therefore, as it is well known, there are {in $N_s$ nonzero classes with $L^2 = K_s\cdot L=0$  only for $s\ge 9$; for $s=9$ these are exactly the multiples of $K_9$, and the ones with non-negative degree form the ray $\langle-K_9\rangle$; and for $s\ge 10$ they form a rational quadratic cone of dimension $s-1$ in $N_s \cong \bbR^{s+1}$.

For $s\ge9$, let $\Qperp \subset N_s$ be the locus of all classes with $L^2=K_s\cdot L=0$ and $H\cdot L\ge 0$ (since we are interested in effectivity and nefness, only classes meeting the  class $H$ nonnegatively are relevant).
By the previous paragraph, $\Qperp[9]$ consists of the single ray spanned by $-K_9$, whereas for $s\ge10$ it is the $H$-nonnegative half of a rational quadratic cone, the boundary of a convex ``round'' cone in the hyperplane $K^\perp$ of classes $L$ with $K_s\cdot L=0$. 

\subsubsection*{The Cremona-Kantor action}
Consider the group $\CK$ generated by quadratic birational transformations of $\bbP^2$
based at subsets of 3 points among the $s$ very general points and by permutations of these points (see \cite{DuV36}, \cite{A02}, \cite{Dol12}).
The Cremona-Kantor group acts on the set of divisor classes  $L_d(m_1,\ldots, m_s)$
preserving all numerical and cohomological properties (nefness, effectiveness, dimension of the associated linear system $\calL_d(m_1,\ldots, m_s)$, etc.).
A divisor class $L$ is \emph{Cremona reduced}
if it has minimal degree in its (CK)-orbit, and for effective classes 
this is the case if and only if the degree 
is greater or equal to the sum of the three largest multiplicities
{
(see \cite[p. 402-403, Thms 8 - 10]{C31};
a modern and more precise treatment may be found in \cite{CC10}).

The canonical class $K_s=-3H +\sum E_i$ is fixed under
the action of $\CK$, so $\CK$ preserves 
(and acts upon) the set $\Qperp$  of all classes of selfintersection
zero orthogonal to $K_s$.

\begin{remark}
	\label{rem:-1-class-cut}
	Let $E$ be the class of a $(-1)$-curve in $\Pic (X_s)$, hence $E^2=K_s\cdot E=-1$.
	It is well known that for $s\ge 3$ all $(-1)$-curves are $\CK$-equivalent (here we use the fact that the points blown up to construct $X_s$ are very general).
		
	Fix $s\ge 9$.
	Each hyperplane in $N_s$ either meets the cone $\Qperp$ only at the origin,  is tangent to it along a ray, or cuts it in two regions (which can only happen for $s\geq 10$). 
	Let us consider the case of $E^\perp$, the hyperplane of classes orthogonal to $E$; because all $E$ are $\CK$-equivalent, which kind of intersection there is between $E^\perp$ and $\Qperp$ depends only on $s$.
	
	Assume for simplicity that $E=E_s$ is the exceptional curve of the last blowup. In that case, a class belongs to $E^\perp$ if and only if it is the pullback to $X_s$ of a class in $X_{s-1}$ (via $X_s\rightarrow X_{s-1}$, the blowing-up of the last point).
	Since pullback preserves intersection multiplicities, $\Qperp \cap E^\perp \cong \Qperp[s-1]$.
	Therefore $E^\perp$ meets $\Qperp$ at the origin for $s=9$, they are tangent along a single ray $\langle -K_9\rangle$ for $s=10$, and they intersect along a lower dimensional irreducible quadric for $s>10$, cutting $\Qperp$ in two regions.
\end{remark}}

\begin{lemma}
	\label{lem:reduced_Kperp}
	Let {$L=L_d(m_1,\ldots, m_s)$} be a class satisfying $L^2=K_s\cdot L=0$
	such that $d$ is greater or equal to the sum of the three largest
	multiplicities. 
	Then $s\ge 9$, and up to reordering of the divisors $E_i$, $L$ is a 
	multiple of {$L_3(1^9,0^{s-9})$}.
\end{lemma}

\begin{proof}
	We will prove that every solution $(d,m_1,\dots,m_s)\in \mathbb{R}^{s+1}$
	to the following set of equations and inequalities:
\begin{itemize}
	\item $d^2 = \sum_{i=1}^s m_i^2$.
	\item $3d =  \sum_{i=1}^s m_i$.
	\item $m_1\ge m_2 \ge \dots \ge m_s \ge 0$.
	\item $d\ge m_1+m_2+m_3$.
\end{itemize}	
	is of the form $(d,m_1,\dots,m_s)=(3m, m^9, 0^{s-9})\in \mathbb{R}^{s+1}$.
	Since $K_s=-3H + \sum_{i=1}^s E_i$, the two equations are equivalent to
	$L^2=0$ and $K_s\cdot L=0$ respectively.
	
	All conditions and the statement are homogeneous in the parameters, so we may assume additionally that $d=1$,  and we shall prove that $(d,m_1,\dots,m_s)=(1, (1/3)^9, 0^{s-9})$. We will need the following auxiliary statement:
	\begin{claim}
	Let $A=m_1+m_2+m_3$.
		For any $a\in [0,m_3]$, 
			$$m_1^2+m_2^2+m_3^2 \le A^2-4aA+6a^2.$$
		Moreover equality holds if and only if $m_1=m_2=m_3=a=A/3$.
	\end{claim}
	\begin{proof}
		 Define $s_2 = m_1m_2 + m_1m_3 + m_2m_3$, the second symmetric function. Expanding
		$A^2$ on the right hand side and cancelling the squares, the claimed inequality is equivalent to
		\begin{equation}
			\label{eq:bound_s_2}
			2s_2 - 4Aa + 6a^2 \ge 0\text{ for all }a \in [0, m_3].		
		\end{equation}
	For fixed $A$, the quantity $s_2$ is minimized when the $m$'s are all equal, i.e., $m_1 = m_2 = m_3 = A/3$; in that case $s_2 = A^2/3$. 
	Since $m_3 \le A/3$, \eqref{eq:bound_s_2} is implied by the statement that
\begin{equation}
\label{eq:bound_A}
2A^2/3 - 4Aa + 6a^2 \ge 0\text{ for all a }\in [0, A/3].
\end{equation}		
	{This is clear, since} the left hand side factors as $(2/3)(A - 3a)^2$, and is always non-negative.
	
		If equality holds, then $s_2$ must achieve its minimum, all the $m$'s are equal to $A/3$, and
		the final inequality must also be an equality, forcing $a = A/3$ as well.
	\end{proof}

Now we can complete the proof of Lemma \ref{lem:reduced_Kperp}.
Since $m_4 \le m_3$, the claim applies with $a=m_4$ and we conclude that
\begin{equation}
	m^2_1 + m^2_2 + m^2_3 \le A^2 - 4Am_4 + 6m^2_4
\end{equation}
with equality holding if and only if $m_1 = m_2 = m_3 = m_4 = A/3$.

By the second equation in the hypotheses, $\sum_{i=4}^s m_i = 3 - A$. Therefore since the $m_i$'s descend, we have 
$$\sum^s_{i=4} m^2_i \le m_4 \sum^s_{i=4} m_i = (3 - A)m_4,$$
so that using the first equation in the hypotheses gives
\begin{equation}
\label{eq:boundAm4}
1 = \sum^s_{i=1} m_i^2 \le A^2 - 4Am_4 + 6m^2_4 + (3 - A)m_4
= A^2 - (5A - 3)m_4 + 6m^2_4.
\end{equation}
This gives a quadratic inequality for $m_4$ that implies that $m_4$ must lie outside the open
interval
$$I = ( \frac{5A - 3 - \sqrt{A^2 - 30A + 33}
}{12} , \frac{5A - 3 + \sqrt{A^2 - 30A + 33}}{12}
)$$
However $m_4$ must be at most $A/3$, and at least $\sum^s_4 m_i/(s - 3) = (3 - A)/(s - 3)$, and
therefore lies in the interval $[(3 - A)/(s - 3), A/3]$.
Now suppose that $A \le 1$. Then a simple calculation shows that the left endpoint of $I$ is
at most zero; since $m_4$ cannot be zero in this case, we must have that $m_4$ is greater than or
equal to the right endpoint of $I$.
However if $A \le 1$, another calculation shows that this right endpoint is at least $A/3$.
This forces $m_4 = A/3$, which forces $m_1 = m_2 = m_3 = m_4 = A/3$ as well. It also
implies that the right endpoint of $I$ is equal to $A/3$; this implies that $A = 1$. Therefore the
first four multiplicities are equal to $1/3$.
Now since $\sum^s_{i=4} m_i = 2$ and $\sum^s_{i=4} m^2_i = 2/3$,
the only way this works is to have 6 of the $m_i$'s for $i \geq 4$ equal to $1/3$ and the rest equal to zero.
\end{proof}

\begin{proposition}
	\label{pro:nefness_Kperp}
	Let {$L=L_d(m_1,\ldots, m_s)$} be a class with $d,m_i \in \bbZ_{\ge 0}$ satisfying $L^2=K_s\cdot L=0$. The following are equivalent:
	\begin{itemize}
		\item $L$ is nef.
		\item $L$ is equivalent to a 
	multiple of {$L_3(1^9,0^{s-9})$} under the action of the Cremona-Kantor group $\CK$.
		\item There is no $(-1)$-curve $E$ with $E\cdot L < 0$.
	\end{itemize}
\end{proposition}

\begin{proof}
	The proof is algorithmic.
	
	If $d=0$, then $L$ is nef if and only if $L=0$, so we may assume that $d>0$.
	
	If $d$ is not smaller than the sum of the three largest $m_i$, then the previous lemma shows that $L$ is a permutation of a multiple of {$L_3(1^9,0^{s-9})$} (and in particular it is nef).
	
	Alternatively, $d$ is less than the sum of the three largest $m_i$.
	Then we may perform a quadratic Cremona map based at the three points of largest multiplicity and the resulting class $L'$ is $\CK$-equivalent to $L$ and has a smaller degree. 
	As long as the degree and multiplicities stay non-negative and $d$ is  less than the sum of the three largest multiplicities, we can replace $L$ by $\CK$-equivalent classes
	$L^{(k)}=L_{d^{(k)}}(m_1^{(k)},\dots,m_2^{(k)})$ with smaller degree.
	The process finishes in one of the following ways:
	\begin{itemize}
		\item One or more of the multiplicities $m_i^{(k)}$ is negative. Then $L^{(k)}$ is not nef, as it intersects the corresponding $E_i$ negatively. The original $L$ is therefore not nef (and the application to $E_i$ of the same quadratic Cremona maps in the reverse order produces a $(-1)$-curve meeting $L$ negatively).
		\item $d^{(k)}<0$. This obviously is also non-nef, and the equality $3 d^{(k)}=\sum m_i^{(k)}$ that follows from $L\cdot K_s=0$ implies that one or more of the multiplicities $m_i^{(k)}$ is negative, so the previous description applies: there is a $(-1)$ curve meeting $L$ negatively.
		\item $d^{(k)}$ is no less than the sum of the three largest multiplicities $m_i^{(k)}$. Then by lemma \ref{lem:reduced_Kperp}, $L^{(k)}$ is a permutation of a multiple of {$L_3(1^9,0^{s-9})$}.
		In particular $L$ is equivalent to a 
		multiple of {$L_3(1^9,0^{s-9})$} under the action of the Cremona-Kantor group $\CK$, and it is nef.\qedhere
	\end{itemize}
\end{proof}

We now consider the non-nef cases, i.e., classes $L=L_d(m_1,\ldots, m_s)$ with $L^2=K_s\cdot L=0$ such that there is a $(-1)$-curve $E$ with $L\cdot E < 0$.

\begin{proposition}[{de~Fernex \cite{DeF01}}]
	\label{pro:kperp-10}
	Let $s=10$.
	Every class $L=L_d(m_1,\ldots, m_s)$ with $d,m_i \ge 0$ satisfying $L^2=K_s\cdot L=0$ is nef.
	
	Moreover, there is a bijection $\phi$ between the set of {rays spanned by} such integral classes $L$ and the set of $(-1)$-curves on $X_{10}$, such that $\phi(L)$ is the unique $(-1)$-curve $E$ with $L\cdot E=0$.
\end{proposition}

We give a proof based on the previous results which seems to us more elementary than de~Fernex's, {cf. \cite[Corollary 4.3]{DeF01}}.

\begin{proof}
	By Remark \ref{rem:-1-class-cut}, for every $(-1)$-curve $E$ the hyperplane $E^\perp$ is tangent to $\Qperp[10]$ along a single ray.
	Therefore, the whole $\Qperp[10]$ is contained in the half-space $E^+$ of classes $L$ with $L\cdot E\ge 0$.
	Thus no class in $\Qperp[10]$ meets any class of a $(-1)$-curve negatively and hence, by Proposition \ref{pro:nefness_Kperp}, all classes on $\Qperp[10]$ are nef.
	
	It remains to give the bijection $\phi$. 
	We have just seen that every integral class $L$ in $\Qperp[10]$ is nef and so, by Proposition \ref{pro:nefness_Kperp}, for every such $L$ there is an element $\sigma$ of the Cremona-Kantor group mapping $L$ to {a multiple of $L_3(1^9,0)$. 
	$E_{10}$ is the only $(-1)$-curve orthogonal to $L_3(1^9,0)$; indeed, every $(-1)$-curve $E$ satisfies $-1=E\cdot K_{10}=-E\cdot L_3(1^{10})$, so in order to be orthogonal to $L_3(1^9,0)$ it must satisfy $E\cdot E_{10}=-1$, which forces $E=E_{10}$.
	Therefore} $\sigma^{-1}(E_{10})$ is the only $(-1)$-curve orthogonal to $L$.
\end{proof}

\begin{proposition}\label{pro:subcones-Kperp}
	Let $s>10$ and let $\mathcal{W}_s$ be the pullback to $X_s$ of all classes in $\Qperp[10]$ via the blowing-up $X_s \rightarrow X_{10}$ of the last $s-10$ points.
	
	{
	For each $(-1)$-curve $E$ on $X_s$, let $\calQ_E$ be the subset of $\Qperp$ formed by all classes $L$ satisfying the inequality $L\cdot E \le 0$.
	Then $\Qperp$ is covered by the subcones $\calQ_E$ (indexed by all $(-1)$-curves $E$ on $X_s$) and each $\calQ_E$ satisfies:}
\begin{enumerate}
	\item Every class in the interior of $\calQ_E$ intersects $E$ negatively and is non-nef.
	\item If $s=11$, all classes on the boundary of $\calQ_E$ are nef. 
	\item If $s>11$, the boundary of $\calQ_E$, namely $\Qperp \cap E^\perp\cong\Qperp[s-1]$, is covered by subcones of smaller dimension, on which nefness is determined by recursively applying this proposition with $s'=s-1$.
\end{enumerate}
	The nef locus $\Nefperp$ on $\Qperp$ { is the topological closure of all $\CK$-translates of the ray $\langle L_3(1^9,0^{s-9})\rangle$, and it coincides with} 
	the union of all $\CK$-translates of $\mathcal{W}_s$, which are 9-dimensional quadratic cones.
	At each rational ray belonging to $\Nefperp$, exactly $s-9$ translates of $\mathcal{W}_s$ meet.
\end{proposition}
	Recall that $\Qperp$ is a $(s-1)$-dimensional quadratic cone, and each $\calQ_E$ is the cone over a ball whose boundary is a $(s-2)$-dimensional quadratic cone.
\begin{proof}
	It is clear that every class in the interior of $\calQ_E$ intersects $E$ negatively, and by Remark \ref{rem:-1-class-cut}, the boundary is isomorphic to the pullback of $\Qperp[s-1]$.
	
	The fact that such cones cover $\Qperp$ follows from Proposition \ref{pro:nefness_Kperp}: every ray on $\Qperp$ is a limit of rational rays, and for every integral $L\in \Qperp$ either there is a $(-1)$-curve $E$ such that $L$ belongs to the interior of $\calQ_E$ or it belongs to the orbit of a multiple of $3H-\sum_{i=1}^9 E_i$, in which case there are $(-1)$-curves $E$ to which it is orthogonal, and so it belongs to the boundary of $\calQ_E$.
	
	Finally, the ray spanned by {$L_3(1^9,0^{s-9})$} is orthogonal to exactly $s-9$ $(-1)$-curves, namely $E_{10},\dots E_{s}$, and so it belongs to $s-9$ translates of $\mathcal{W}_s$, which we call {$\mathcal{W}_s^{(1)}$, $\mathcal{W}_s^{(2)}$, \dots, $\mathcal{W}_s^{(s-9)}$}.
	Since every $L\in \Nefperp$ is the translate of a multiple of {$L_3(1^9,0^{s-9})$} by some $\sigma \in \CK$, it belongs to $s-9$ distinct translates of $\mathcal{W}_s$, namely  {$\sigma^{-1}(\mathcal{W}_s^{(1)})$, \dots, $\sigma^{-1}(\mathcal{W}_s^{(s-9)})$}. 
\end{proof}

\begin{proof}[Proof of Theorem \ref{thm:main-kperp}]
	If $s=10$, the claims are equivalent to Proposition \ref{pro:kperp-10}
	and also follow from Proposition \ref{pro:nefness_Kperp}. 
	If $s>10$, we use Proposition \ref{pro:subcones-Kperp} and induction on $s$. 
	Indeed, if $L$ belongs to the interior of some cone $\calQ_E$ then it clearly does not satisfy conditions (1) and (3) in the statement of the Theorem. 
	It does not satisfy (2) either: each  collection of $s-10$ disjoint $(-1)$-curves $C_1, \dots, C_{s-10}$ can be blown down $X_s\rightarrow X'$ to a surface $X'$ isomorphic to $X_{10}$ (because it is a generic rational surface with Picard number $10$, see also \cite[Corollary 2.5]{DeF01}) and $L$ being orthogonal to them would imply that $L$ is the pullback to $X_s$ of a well-defined class in $\Qperp[10]$, therefore nef by Proposition \ref{pro:kperp-10}.
	Thus we may assume that $L$ belongs to the boundary of $\calQ_E$ for some $E$, in which case we may blow down $E$; the result is isomorphic to $X_{s-1}$, and there is a class $L_{s-1}$ on $X_{s-1}$ whose pullback to $X_s$ is $L$. 
	It is easy to see that $L_{s-1}^2=L_{s-1}\cdot K_{s-1}=0$ and the equivalence of (1)-(3) follows by induction.
	
	Finally, since each  collection of $s-10$ disjoint $(-1)$-curves $C_1, \dots, C_{s-10}$ can be blown down to $X'\cong X_{10}$, and the classes $L$ with $C_i\cdot L=0$ for $i=1,\dots,s-10$ cover the image $M$ of the pullback map $N(X')\rightarrow N(X_s)$, which is a linear isomorphism onto its image, it follows that the subset cut out on $M$ by the additional equations $L^2=K\cdot L=0$ is isomorphic to $\Qperp[10]$, namely a 9-dimensional rational quadratic cone.
\end{proof}

\section{Collision of $r^2$ points}
\label{sec:collision}

Our method to construct round parts of the boundary of the Mori cone which are not orthogonal to the anticanonical divisor relies on a degeneration of $X_s$ where $r^2$ of the $s$ blown-up points, of equal multiplicity $m$ in certain divisor class 
$$L=L_d(m^{r^2},m_{r^2+1},\dots,m_s),$$
come together. 
This kind of \emph{collision} was explained in detail in \cite{CMR23} using the technique of \cite{CM98}, and we 
give a brief sketch of what we need here.

We consider a trivial family $\calX = X_{s-r^2} \times \Delta$ over a disc $\Delta$,
and blow up a general point in the central fiber over $0\in \Delta$ 
to obtain the threefold $\calX'$.
This produces a degeneration of $X_{s-r^2}$
to a union of two surfaces, 
a plane (the exceptional divisor for the blowup)
and the proper transform $F$ of the original $X_{s-r^2}$ fiber,
which is now isomorphic to $X_{s-r^2+1}$.
These two surfaces intersect transversely along a smooth rational curve $R$
which is a line in the plane
and a $(-1)$-curve in $F$.

We now choose $r^2$ general points on the plane;
extend these $r^2$ general points to the general fiber 
using $r^2$ sections of the projection of $\calX'$  to $\Delta$,
and blow up those $r^2$ sections to ruled surfaces
$\mathcal E_1,\ldots,\mathcal E_{r^2}$.
This then produces a threefold $\calY$
which is a degeneration of $X_s$,
to a union of a surface $P \cong X_{r^2}$ and $F \cong X_{s-r^2+1}$,
intersecting transversely along the double curve $R$.
This smooth rational curve $R$
is the pullback of a general line in the surface $P$
and remains a $(-1)$-curve in the surface $F$.

We have the line bundle $\calO_{X_{s-r^2}}(L')$ corresponding to 
$L'=L_d(m_{r^2+1},\dots, m_s)$
on $X_{s-r^2}$, and can extend it trivially to $\calX$.
If we pull that back to the first blowup $\calX'$,
we see that this restricts to the bundle corresponding to
$L_d(m_{r^2+1},\dots,m_s,0)$ on the surface $F\cong X_{s-r^2+1}$,
and to the trivial bundle on the plane.
We then pull that back to the second blowup $\calY$,
and tensor by $\calO_{\calY}(-tP-m \sum_{i=1}^{r^2} \calE_i)$, 
with $t$ a non--negative integer (called the \emph{twisting parameter}).
This produces a line bundle $\calM$ on $\calY$,
which restricts to the general fiber as the original bundle
$\calO_{X_s}(L)$.
The principle of semicontinuity guarantees that
the dimension of the general linear system 
$\calL_d(m^{r^2},m_{r^2+1},\ldots,m_s)$
is at most equal to the dimension of the linear system on the reducible surface $P+F$,
and in particular $L$ is not effective as soon as $\calM|_{P\cup F}$ is not effective, i.e., it has no nonzero sections,
for some choice of twisting parameter $t$.

Fix $t=mr$ as twisting parameter.
The restrictions of $\calM$ to $P\cong X_{r^2}$  and 
to $F\cong X_{s-r^2+1}$ are
\begin{align*}
\calM|_P=&\, \calO_{X_{r^2}}(L_{mr}(m^{r^2})),\\
\calM|_F=&\, \calO_{X_{s-r^2+1}}(L_d(rm, m_{r^2+1},\dots, m_s)),
\end{align*}
respectively.
The space of global sections of $\calM|_{P\cup F}$ is the fiber product
of the space of sections on $P$ with the space of sections on $F$,
fibered over the restriction to the space of sections on $R$. Therefore to prove that $L$ is not effective it suffices to prove that this fibre product is zero. As an application, we have:

\begin{lemma}\label{lem:uncollision}
Fix $r\ge 2$ and multiplicities $m$, $m_{r^2+1},\ldots,m_s$. If either
\begin{itemize}
\item[(a)] $r=2$ and $h^0(\calO_{X_{s-3}}(L_d(2m,m_{5},\dots,m_s)) \le m$, or
\item[(b)] $r\ge 3$ and $h^0(\calO_{X_{s-r^2+1}}(L_d(rm, m_{r^2+1},\dots,m_s))) \le 1$,
\end{itemize}
then $L_d(m^{r^2},m_{r^2+1},\dots,m_s)$ is non--effective.
\end{lemma}

\begin{proof}
	
Cases $r=2,3$ are parts (a) and (b) from \cite[Lemma 2]{CMR23}. 
For $r\ge4$,  recall that, as explained in \cite[end of section 1]{CMR23}, Nagata's theorem guarantees that if $r\ge4$ then all classes of the form $L_{rm}(m^{r^2})$ are non--effective; thus the bundle on $P$ has no nonzero sections, and therefore the fibre product corresponds to the subsystem $\mathcal L_d(rm+1, m_{r^2+1},\dots,m_s)$ on $F$. So we have to prove that this is empty. Now, by \cite [Proposition 2.3]{CC}, this is a proper subsystem of $L_d(rm,m_{r^2+1},\dots,m_s)$ and therefore it is empty (indeed, if $\mathcal L_d(rm, m_{r^2+1},\dots,m_s)$ is non--empty, its general member has the first point of multiplicity exactly $rm$). 
\end{proof}

We note that the divisors before and after the collision
have the same self-intersection.
In particular, if one is zero, so is the other;
this will be important in our application.

Given a divisor class $L\in \Pic(X_{s+1})$, and an index $i$ denoting one of the multiplicities,
in \cite{CMR23} we defined the \emph{uncollision} $\Uncoll_r(L,i)$
as the class on $\Uncoll_r(L,i)\in \Pic(X_{s+r^2})$ obtained replacing the $i$-th multiplicity $m_i$
by $r^2$ points of multiplicity $m_i/r$.
This makes sense at the level of divisor classes if $m_i$ is divisible by $r$, 
but it also makes sense as a map $N(X_{s+1})\rightarrow N(X_{s+r^2})$.
Observe as well that the process of considering an uncollision behaves linearly with respect to multiplicities and degrees, and the corresponding linear map is injective. 
This will be key in our application, 
and additionally it means that uncollision makes sense applied to rational rays in $N(X_{s+1})$.

\section{Quadratic sections of $\partial \overline{NE}$ in $K_s^+$}
\label{sec:round-parts}

By the Cone Theorem, the shape of the Mori cone on the half-space $K_s^-$ of classes which intersect the canonical divisor negatively is governed by the rays generated by $(-1)$-curves. 
On the orthogonal hyperplane $K_s^{\perp}$ this is no longer quite the case, 
but we saw in the previous section that nefness is still characterized by intersection with $(-1)$--rays, and there are no good rays on $K_s^{\perp}$.

In this section we show how to exploit uncollisions to build {9-dimensional quadric cones} in the boundary of the Mori cone consisting entirely of good and wonderful rays.
Since the collision/uncollision analysis and construction is only available for {divisor classes or rational rays} but not for irrational rays, we work on rational rays to obtain good rays and
then use the closed convex nature of the nef cone to obtain families of good and wonderful rays.

\begin{remark}
\label{rem:uncollision_K}
If $D$ is an $\bbR$-divisor class with $D \cdot K_s=0$ and $D'$ is obtained from $D$ by uncolliding a point of multiplicity $rm>0$ to $r^2\ge 4$ points of multiplicity $m$, then $D' \cdot K_s >0$.
Indeed, writing $D=dL-\sum m_i E_i$ we have $D\cdot K_s=\sum m_i -3d$ and
\[
D'\cdot K_s = \sum m_i - rm + r^2m -3d = D \cdot K_s + (r^2-r)m>D\cdot K_s.
\]  
\end{remark}

{
	\begin{proposition}\label{pro:uncoll_kperp}
	Let $L$ be a nef class in $\Qperp$. 
	For every $i=1,\dots,s$ such that $m_i\ne 0$ and every $r\ge2$ the ray
	$\langle \Uncoll_r(L,i) \rangle$ is good.	
\end{proposition}

\begin{proof}
	By Proposition \ref{pro:subcones-Kperp}, $L$ is $\CK$-equivalent to $L_{3a}(a^9,0^{s-9})$ for some integer $a>0$, and therefore its space of global sections has dimension 1. 
	Therefore by Lemma \ref{lem:uncollision} $\Uncoll_r(L,i)$ is not effective, and the same holds for its multiples, which are uncollisions of multiples of $L$, which themselves have 1-dimensional global sections.
\end{proof}

\begin{proof}[Proof of Theorem \ref{thm:main-kplus}]
Let $s'=s-3\ge 10$.
By Theorem \ref{thm:main-kperp}, each collection of $s'-10$ disjoint $(-1)$-curves $C_1, \dots, C_{s'-10}$ determines a $10$-dimensional linear subspace $\Pi'\subset N_{s'}$ (of those classes orthogonal to $K_{s'}$ and to the $C_i$) such that each $D\in \Pi'$ with $D^2=0$ is nef, and the rays spanned by such $D$ cover a 9-dimensional rational quadratic cone.
By Proposition \ref{pro:uncoll_kperp}, for each $i=1,\dots,s'$, every such class $D$ gives rise by uncollision to a good ray $\langle\Uncoll_2(D,i)\rangle$ in $N_s$ of self--intersection zero. Consider $\Uncoll_2(\cdot,i)$ as a linear map $\Pi'\rightarrow N_s$ and let $\Pi$ be its image. By linearity and injectivity, the good rays obtained as images of the nef rays $\langle D\rangle$ with $D^2=0$ cover the cone
$\{L\in \Pi \,|\, L^2=0, H\cdot L\ge 0\},$
and therefore $\Pi\cap \overline{NE}(X_s)=\{L\in \Pi\,|\,L^2\ge 0, H\cdot L\ge 0\}$.

By Remark \ref{rem:uncollision_K}, all rational classes $L$ on the cone $\calC=\{L\in \Pi\,|\,L^2= 0\}$ satisfy $L\cdot K_s>0$. 
Now $\calC^\perp := \calC\cap K_s^\perp$ is the intersection of a rational quadratic cone with a rational hyperplane, thus either it is a rational quadratic cone, or it consists only of the single point at the origin. 
However, we know that $\calC^\perp$ contains no rational ray, so it must be reduced to a point, and we conclude that $L\cdot K_s> 0$ for every nonzero $L \in \calC$.
\end{proof}

The 9-dimensional quadratic cones of good and wonderful rays in $K_s^\perp$ we just constructed do not in general consist of de~Fernex negative classes. 
However, the wonderful de~Fernex negative rays constructed in \cite{CMR23} are uncollisions of nef rays in $K_s^\perp$, so they do belong to some of these 9-dimensional cones, which is the basis for the proof of Theorem \ref{thm:main-fminus}.

\begin{proof}[Proof of Theorem \ref{thm:main-fminus}]
	Proposition 17 in \cite{CMR23} exhibits de~Fernex negative wonderful rays $\langle D\rangle$ on  {$X_{k^2+4}$ for every $k\ge 3$ by uncolliding classes $L$ on $\Qperp[2k+4]$ (use $n=k-2$ in \cite[Proposition 17]{CMR23}).
	By Proposition \ref{pro:subcones-Kperp}, for every such $L$ there is a rational 9-dimensional quadratic cone $\calC$ of nef classes in $\Qperp[2k+4]$ containg $L$.
	Therefore, by Proposition \ref{pro:uncoll_kperp}, the uncollision $\Uncoll_{k-1}(L,1)$} spans a good ray for every $L$ in $\calC$; these rays cover the 9-dimensional cone over a 8-dimensional sphere and at least one such ray is de~Fernex negative. 
	Since being de~Fernex negative is an open condition, the claim follows.	
\end{proof}
}

\end{document}